\documentclass{amsart}
\usepackage{amsmath, amsthm, amsfonts, amssymb}
\usepackage{mathrsfs}
\usepackage{xcolor}
\usepackage[numbers,sort&compress]{natbib}
\newtheorem{theorem}{Theorem}[section]
\newtheorem{lemma}[theorem]{Lemma}

\theoremstyle{definition}
\newtheorem{definition}[theorem]{Definition}

\newtheorem{proposition}[theorem]{Proposition}
\newtheorem{remark}[theorem]{Remark}

\numberwithin{equation}{section}
\def\w{\omega}

\def\v {\varepsilon}

\def\a {\alpha}
\def\d {\delta}
\def\g {\gamma}
\def\m {\mu}
\def\r+ {\mathbb{R}^+}
\def\r {\mathbb{R}}

\begin{document}

\title[A note on Nakai's generalized Morrey spaces and applications]
{Morrey meets Muckenhoupt: A note on Nakai's generalized Morrey spaces and applications}

 \author{Xian Ming Hou}
\address{Xianming Hou\\School of Mathematical Sciences\\Xiamen
University \\
Xiamen 361005\\
 P. R. China}
\email{houxianming37@163.com}
%
\author{Qing Yan Wu}
\address{Qingyan Wu\\Department of Mathematics\\Linyi
University\\ Linyi 276005, P.R. China and Department of Mathematics, Macquarie University, NSW, 2109, Australia}
\email{wuqingyan@lyu.edu.cn}

  \author{Zun Wei Fu}
\address{Zunwei Fu\\Department of Mathematics\\Linyi
University \\
Linyi 276005\\
P. R. China and Department of Information and Telecommunications Engineering, The University of Suwon, Wau-ri, Bongdam-eup, Hwaseong-si,
Gyeonggi-do, 445-743, Korea} \email{zwfu@mail.bnu.edu.cn or zwfu@suwon.ac.kr}

\author{Shan Zhen Lu}
\address{Shanzhen Lu\\School of Mathematical Sciences, Beijing Normal University, Beijing,
100875, P. R. China}
\email{lusz@bnu.edu.cn}
\thanks{}
\subjclass[2010]{42B35, 42B20, 42B99}


\keywords{Morrey space, one-sided weight, Riemann-Liouville fractional integral, fractional differential equation.}

\begin{abstract}
The goal of this paper is to extend Nakai's generalized Morrey spaces to a wider function class, the one-sided Muckenhoupt weighted case. Morrey matching Muckenhoupt enables us to study the weak and strong type boundedness of one-sided sublinear operators satisfying certain size conditions on the one-sided weighted Morrey spaces.
Meanwhile, the boundedness and compactness of Riemann-Liouville integral operators on locally one-sided weighted Morrey space are considered. As applications, we establish the existence and uniqueness of solutions to a Cauchy type problem associated with fractional differential equations.
\end{abstract}
\maketitle

\section{Introduction}
It is well-known that Morrey \cite{MO} first introduced the classical Morrey spaces $L^{p,\lambda}(\Omega)$ to investigate the local behavior of solutions to second-order elliptic partial differential equations:
\begin{align*}
\|f\|_{L^{p,\lambda}(\Omega)}:= {\sup_{x\in \Omega, r>0}} \left(\frac{1}{|B(x, r)|^{\lambda}} \int_{B(x, r)\cap\Omega}|f(y)|^p dy \right)^{1/p},
\end{align*}
where $1\leq p<\infty$, $\lambda\geq 0$, $\Omega \in \mathbb{R}^{n}$ is an open set. It is obvious that $L^{p, 0}(\Omega)=L^p(\Omega)$ and $L^{p, 1}(\Omega)=L^\infty(\Omega)$.
Therefore, it is interesting to consider the case of $0<\lambda<1$.

As we know, Morrey space is nonseparable. Therefore, there are no approximation skills in this kind of spaces, which are quite different from the separable spaces (e.g. continuous or integrable functional spaces). In recent years, there is an explosion of interest in Morrey type spaces and their applications, see \cite{Rafeiro}, \cite{YSY} and references therein.
In \cite{Mizuhara}, Mizuhara studied a version of generalized Morrey spaces and got the boundedness of Calder\'{o}n-Zygmund operators on these spaces. Nakai \cite{Nakai} introduced the following generalized Morrey space $L^{p,\omega}(\mathbb{R}^n)$:

Let $1\leq p<\infty$ and denote $\mathcal{Q}=\mathcal{Q}(\mathbf{a},r)=\{x\in\mathbb{R}^n: |x_i-a_i|\leq r/2, i=1,2,\ldots,n\}$ for every $r>0$
$$L^{p,\omega}(\mathbb{R}^n):= \left\{f\in L_{\rm { loc}}^p(\mathbb{R}^n):\sup_{\mathcal{Q}}\frac{1}{\omega(\mathcal{Q})} \int_{\mathcal{Q}}|f(x)|^p dx <\infty \right\},$$
$$\|f\|_{p,\omega}:= \sup_{\mathcal{Q}}\left(\frac{1}{\omega(\mathcal{Q})} \int_{\mathcal{Q}}|f(x)|^p dx\right)^{1/p},$$
where $\omega: \mathbb{R}^n\times\mathbb{R}^+\rightarrow \mathbb{R}^+$ and for $\mathbf{a}\in \mathbb{R}^n$, $\omega$ satisfies
\begin{align}\label{1.1}
r\leq t\leq2r \Longrightarrow C^{-1}\leq\frac{\omega(\mathbf{a},t)}{\omega(\mathbf{a},r)}\leq C,
\end{align}
\begin{align}\label{2.1}
\int_r^\infty\frac{\omega(\mathbf{a},t)}{t^{n+1}}dt\leq C \frac{\omega(\mathbf{a},r)}{r^{n}}.
\end{align}
The boundedness of Hardy-Littlewood maximal operator, singular integral operators and the Riesz potentials on $L^{p,\omega}(\mathbb{R}^n)$ was established.
Then $L^{p,\omega}(\mathbb{R}^n)$ is a Banach space with norm $\|\cdot\|_{p,\omega}$. If $\omega(\mathbf{a},r)=1$, then $L^{p,\omega}= L^p$. If $\omega(\mathbf{a},r)=r^n$, one has $L^{p,\omega}=L^\infty$. And for $0<\lambda<1$, $\omega(\mathbf{a},r)=r^{n\lambda}$, $L^{p,\omega}$ is the classical Morrey space $L^{p,\lambda}$.
The article \cite{Nakai} of Nakai has been cited nearly three hundred times and attracts a great deal of interest.
Softova \cite{SOF} studied singular integrals and commutators in generalized Morrey spaces. Guliyev et al. \cite{Guliyev} considered the boundedness of the maximal, potential and singular operators in the generalized Morrey spaces. The action of the generalized fractional integral operators and the generalized fractional maximal operators was investigated in the framework of Morrey spaces in \cite{SAWST}. In \cite{Ne}, Nakai proved basic properties of Orlicz-Morrey spaces and gave a necessary and sufficient condition for boundedness of the Hardy-Littlewood maximal operator on Orlicz-Morrey space.
In \cite{LuYangZhou}, Lu et al. established the boundedness of rough operators and their
commutators with BMO functions in generalized Morrey spaces. Then, Komori and Shirai introduced the weighted Morrey space in \cite{KS}.
See \cite{BGG,
BGGM, CDW, GUL, FanLuYang, FanPanYang, Kurata1, Mizuhara1, Persson, Samko, SAW, SW} for more information on this topic.

On the other hand, both the generalization of the theory of two-sided operators and the requirements
of ergodic theory (see \cite{Ca}, \cite{JKGW}, \cite{MOD}) are motivations to study one-sided estimates for operators.
The well-known Riemann-Liouville fractional integral can be viewed as the one-sided
version of Riesz potential, see \cite{BW}.
The study of one-sided operators promote naturally the
development of one-sided spaces. See \cite{FLSP,FLSS,MR,RD,SA} for other works about one-sided operators and spaces. It is shown that many famous results in harmonic analysis still hold for a class of smaller operators (one-sided operators) and a class of wider weights (one-sided weights). It should be pointed out that, in one-sided cases, the classical reverse H\"{o}lder's inequality and doubling condition do not hold.
Thus, these new changes have caused some difficulties in our proofs.

Let $\omega$ be an $A_p$ weight function (see \cite{Muckenhoupt,Muckenhoupt1}). If $0<\alpha<1$, $1\leq p\leq 1/\alpha$, $\omega(I)=\left(\int_I \omega(x)dx\right)^\alpha$, then $\omega$ satisfies \eqref{1.1} and \eqref{2.1}.
As is well-known, the $A_p$ weight functions satisfy doubling conditions \eqref{1.1}.
However, some one-sided $A_p$ weights do not satisfy \eqref{1.1}. In fact, let $1\leq p<\infty$. One-sided $A_p$ weights $\omega$ satisfy a so-called one-sided doubling condition:
\begin{align}\label{11.1}
\int _{x_0-h}^{x_0+h}\omega \lesssim \int _{x_0}^{x_0+h}\omega
 \end{align}
for all $x_0\in \mathbb{R}$ and $h>0$. See its definition in Section 2. We should mention that the one-sided weight function $\omega(x)= e^x$ verifies \eqref{11.1}, but does not satisfy \eqref{1.1} and \eqref{2.1}. Our goal is to define one-sided weighted Morrey spaces and give some applications.

This paper will be organized as follows. In Section $2$, we establish the weak and strong type boundedness of one-sided sublinear operators under specific size conditions and some preliminaries that are essential to our proofs. Section $3$ is devoted to the boundedness of Riemann-Liouville integral operators on locally one-sided weighted Morrey space. The representation of functions by fractional integrals in locally one-sided weighted Morrey spaces will be given in Section $4$. In Section $5$, by Schauder's fixed-point theorem, we obtain the existence and uniqueness of solutions to a Cauchy type problem for fractional differential equations by a new weighted Fr\'{e}chet-Kolmogorov theorem.

Throughout this paper, for $x_0 \in  \mathbb{R}$ and $h,\;\lambda > 0$, we will always denote that $I = (x_0,x_0+h)$, $I^+ = (x_0+h,x_0+2h)$, $I^- = (x_0-h, x_0)$, $\lambda I = (x_0,x_0+\lambda h)$ and $(\lambda I)^-_{-}= (x_0-(\lambda+1)h,x_0-h)$. $C$ stands for a positive constant but it may change from line to line. We sometimes use the notation $a \lesssim b$ as an equivalent of $a\leqslant Cb.$
$\chi_{\Omega}$ represents the characteristic function of set $\Omega.$
%

\section{one-sided weighted Morrey spaces and related sublinear operators}

We will describe the process of defining of one-sided weighted Morrey space in brief.

Firstly, let $1<p<\infty$ and $0\leq\lambda<1$. We define
$$\|f\|_{\widetilde{L}^{p,\lambda}(\mathbb{R})}:=\sup_{x_0\in\mathbb{R},h>0}\left(\frac{1}{h^{\lambda}}
\int_{x_0}^{x_0+h}|f(y)|^pdy\right)^{1/p}<\infty.$$\\
Set $x_0=z_0-h/2$ and $h/2=h'$. By a simple estimate, we have $\widetilde{L}^{p,\lambda}(\mathbb{R})=L^{p,\lambda}(\mathbb{R}).$

Secondly, we equip the Lebesgue measure on interval $(x_0, x_0+h)$. Then $h^\lambda=|(x_0, x_0+h)|^\lambda.$  We deduce
$$\|f\|_{L^{p,\lambda}(\mathbb{R})}:=\sup_{x_0\in\mathbb{R},h>0}\left(\frac{1}{|(x_0, x_0+h)|^\lambda}\int_{x_0}^{x_0+h}|f(y)|^pdy\right)^{1/p}<\infty.$$\\

Finally, for $1\leq \theta<\infty$, we change the Lebesgue measure on interval $(x_0, x_0+h)$ into the weight function measure $\omega$ and let $\omega^\theta(x_0,x_0+h)=\int_{x_0}^{x_0+h}\omega^\theta$. Clearly, when $\omega \equiv 1$, $h^{\lambda-1}\omega^\theta(x_0, x_0+h)=h^\lambda$.

Let $1\leq p<\infty$, $0\leq\lambda< 1$, $1\leq \theta<\infty$. For any $x_0\in \r$ and $h>0$, set
\begin{align}\label{eq2.1}
\Phi_{\omega,\lambda,\theta}^{+}(x_0,h):={{h^{\lambda-1}}\omega^\theta(x_0-h,x_0)},\,
\Phi_{\omega,\lambda,\theta}^{-}(x_0,h):={{h^{\lambda-1}}\omega^\theta(x_0,x_0+h)}.
\end{align}
Now we give the definition of one-sided weighted Morrey space as follows:
$$\|f\|_{{L^{p,\lambda}_{+}(\r,\omega^\theta)}}:= \sup_{{x_0\in\r
, h>0}}\Big(\frac{1}{\Phi_{\omega,\lambda,\theta}^{+}(x_0,h)}\int_{x_0}^{x_0+h}
|f(y)|^pdy\Big)^{1/p}<\infty,$$
and $$\|f\|_{{L^{p,\lambda}_{-}(\r,\omega^\theta)}}:= \sup_{{x_0\in\r
, h>0}}\Big(\frac{1}{\Phi_{\omega,\lambda,\theta}^{-}(x_0,h)}\int_{x_0}^{x_0+h}
|f(y)|^pdy\Big)^{1/p}<\infty.$$
We denote $\|f\|_{{L^{p,\lambda}_{+}(\mathbb{R},\omega^\theta)}}=\|f\|_{{L^{p,\lambda}_{+}(\omega^\theta)}}$.

Let $1\leq p<\infty$, $0\leq\lambda< 1$, $1\leq \theta<\infty$. We say a function belongs to the one-sided weighted Morrey spaces of weak type if one of the following norms is finite:
$$\|f\|^{p}_{W{L^{p,\lambda}_+(\r,\omega^\theta)}}:= \sup_{{x_0\in\r,h>0}}\frac{1}{\Phi_{\omega,\lambda,\theta}^{+}(x_0,h)}\sup_{\gamma>0 }\gamma^p\big|\left\{x\in(x_0,x_0+h):|f(x)|>\gamma\right\}\big|,$$
$$\|f\|^{p}_{W{L^{p,\lambda}_-(\r,\omega^\theta)}}:= \sup_{{x_0\in\r,h>0
}}\frac{1}{\Phi_{\omega,\lambda,\theta}^{-}(x_0,h)}\sup_{\gamma>0 }\gamma^p\big|\left\{x\in(x_0-h,x_0):|f(x)|>\gamma\right\}\big|.$$

We turn to the one-sided weight functions now. In 1986, Sawyer \cite{SA} first introduced the one-sided Muckenhoupt weights $A_p^+$ and $A_p^-$ to treat the one-sided Hardy-Littlewood maximal operators
$$M^+f(x):= \sup\limits_{h>0}\frac{1}{h}\displaystyle\int_x^{ x+h}|f(y)| dy,\quad
M^-f(x):= \sup\limits_{h>0}\frac{1}{h}\displaystyle\int_{ x-h}^x|f(y)| dy.$$
A positive function $\omega$ is said to belong to $A_p^+$ or $A_p^-$ if it satisfies
$$A_p^+(\omega) :=\sup _{a<b<c}\frac{1}{(c-a)^p}\int_a^b\omega(x)dx\left(\int_b^c\omega(x)^{ 1-p'}dx\right)^{p-1} < \infty$$
or
$$A_p^-(\omega):=\sup _{a<b<c}\frac{1}{(c-a)^p}\int_b^c\omega(x)dx\left(\int_a^b\omega(x)^{ 1-p'}dx\right)^{p-1} < \infty,$$\\
when $1< p < \infty$; also, for $p = 1$,
\begin{center}
$A_1^+:M^-\omega \leq C\omega,$ \quad$A_1^-:M^+\omega \leq C\omega,$
\end{center}
for some constant $C$. If $1\leq p<\infty$, then $A_p\subsetneqq A_p^+$ and $A_p\subsetneqq A_p^-$. Notice that the function $\omega(x)=e^x$ mentioned above is in $A_p^+$ but not in $A_p$.

Similarly, the double weight classes  $A^+_{(p,q)}$ and $A^-_{(p,q)}$ are denoted by
$$A^+_{(p,q)}:  \frac{1}{(c-a)^{1-\alpha}}\left(\int_a^b\omega^q\right)^{1/q}\left(\int_b^c\omega^{ -p'}\right)^{1/p'} < C,$$
$$A^-_{(p,q)}:  \frac{1}{(c-a)^{1-\alpha}}\left(\int_b^c\omega^q\right)^{1/q}\left(\int_a^b\omega^{ -p'}\right)^{1/p'} < C$$
for all $a<b<c \in \mathbb{R}$, $0<\alpha<1$, $1<p<q$ and $1/p-1/q=\alpha$; also for $p=1$, $1-1/q=\alpha$,
\begin{center}
$A_{(1,q)}^+:M^-{\omega}^q \lesssim {\omega}^q ,$\quad
$A_{(1,q)}^-:M^+{\omega}^q \lesssim {\omega}^q .$
\end{center}

In \cite{AFM}, Aimar, Forzani and Mart\'{i}n-Reyes introduced the one-sided Calder\'{o}n-Zygmund singular integrals defined by:
$$T^+f(x) = \lim_{\varepsilon \rightarrow 0^+} \int_{x+\varepsilon} ^\infty K(x-y)f(y)dy$$and
$$T^-f(x) = \lim_{\varepsilon \rightarrow 0^+} \int_{-\infty} ^{x-\varepsilon} K(x-y)f(y)dy,$$
where the kernel $K$ is called the one-sided Calder\'{o}n-Zygmund kernel (OCZK) which satisfies
\begin{equation}\label{(1.3)}
\left|\int_{a<|x|<b}K(x)dx\right| \leq C, \quad0 < a < b,
\end{equation}
\begin{equation}\label{(1.4)}
|K(x)|\leq C/|x|, \quad x\neq 0,
\end{equation}
\begin{equation}\label{(1.5)}
|K(x-y)-K(x)| \leq C|y|/|x|^2, \quad|x| > 2|y|>0,
\end{equation}
with support in $ \mathbb{R}^-= (-\infty, 0)$ or $ \mathbb{R}^+ = (0, +\infty)$, where $(\ref{(1.5)})$ is named H\"{o}rmander's condition. Equation $(\ref{(1.4)})$ is also called the size condition for $K$. An interesting example is
$$K(x) = \frac{\sin(\log|x|)}{(x\log|x|)}{\chi_{(-\infty,0)}}(x),$$
where $ \chi_E$ denotes the characteristic function of a set $E$, for more details one can refer to \cite{AFM}.

The one-sided $A_p^+ $ classes not only control the boundedness of one-sided Hardy-Littlewood maximal operators, but also serve as the right weight classes for one-sided singular integral operators. Set $1 < p < \infty$ and let $K$ be an OCZK with support in $\mathbb{R}^-$. Then $T^+$ is bounded on $L^p(\omega)$ if $\omega \in A_p^+ $, see \cite{AFM}.

Fractional integral operators also play an important role in harmonic analysis.
Suppose $0<\alpha<1$. The one-sided fractional maximal operators and the one-sided fractional integrals are defined by
$$M^+_\alpha f(x) = \sup\limits_{h>0}\frac{1}{h^{1-\alpha}}\displaystyle\int_x^{ x+h}|f(y)| dy,\quad
M^-_\alpha f(x) = \sup\limits_{h>0}\frac{1}{h^{1-\alpha}}\displaystyle\int_{ x-h}^x|f(y)| dy,$$
$$I^+_\alpha f(x) =\int_x^\infty\frac{f(y)}{(y-x)^{1-\alpha}}dy,\quad I^-_\alpha f(x) =\int_{-\infty}^x\frac{f(y)}{(y-x)^{1-\alpha}}dy,$$
respectively. The one-sided fractional integrals are also called the Riemann-Liouville and Weyl fractional integral operators. For $A^+_{(p,q)}$ classes, Andersen and Sawyer \cite{AS} obtained the boundedness of $M^+_\alpha$ on weighted Lebesgue spaces.

In recent years, there are many works on convolution-type one-sided operators. In this paper, we want to study a kind of one-sided operators with nonconvolution kernels. We can adopt a definition made in a previous study to introduce  one-sided sublinear operator $\mathcal{T}^+$ and $\mathcal{T}^+_\alpha$$(0<\alpha<1)$ under certain weak size conditions, see \cite{LY,SW,SF}.
The size condition can be introduced:
\begin{equation}\label{(1.33)}
|\mathcal{T}^+f(x)| \lesssim \int_{x}^\infty\frac{|f(y)|}{y-x}dy, \quad \rm{for\:all }\:\emph x\in \r.
\end{equation}
We also define the corresponding fractional size condition as follows:
\begin{equation}\label{(1.35)}
~~~~~~~|\mathcal{T}^+_\alpha f(x)|\lesssim\int_{x}^\infty\frac{|f(y)|}{(y-x)^{1-\alpha}}dy, \quad0<\alpha<1,\,x\in \mathbb{R}.
\end{equation}

It is easy to check that the condition $(\ref{(1.33)})$ is satisfied by many one-sided operators, such as
$M^+$, $T^+$, one-sided dyadic Hardy-Littlewood maximal operator  \cite{LM}, one-sided oscillatory singular integrals  \cite{FLSP,FLSS}, one-sided singular integrals with H\"{o}rmander type kernels \cite{Sh}
 and so on. Both $M^+_\alpha$ and $I^+_\alpha$  satisfy $(\ref{(1.35)})$. Corresponding to \eqref{(1.33)} and \eqref{(1.35)}, we can also define
$\mathcal{T}^-$ and $\mathcal{T}^-_\alpha $. For simplicity, here we omit the corresponding parts.

We now give the endpoint estimates of $\mathcal{T}^+$ on one-sided weighted Morrey spaces.

\begin{theorem} \label{Theorem 1.12}\,\,

Let $0\leq \lambda<1$, $\omega\in A_1^+$, and the one-sided sublinear operator $\mathcal{T}^+$ satisfies $(\ref{(1.33)})$. If $\mathcal{T}^+$ is bounded from $L^1$ to $L^{1,\infty}$, then $\mathcal{T}^+$ is bounded from $L_+^{1,\lambda}(\omega)$ to $WL_+^{1,\lambda}(\omega)$.
\end{theorem}

Covering the shortage in Proposition 12 of \cite{SF}, we conclude that

\begin{theorem} \label{Theorem 6.12}\,\,
Let $0<\alpha<1$, $0\leq\beta,\lambda<1$, $1\leq p,q<\infty$, $\beta/p=\lambda/q$, $1/p=1/q+\alpha$ and
$\alpha-1/q+(\beta-1)/p+1<0$.\\
{\rm{(i)}} If $1<p<\infty$ and $\mathcal{T}^+_\alpha$ is bounded from $L^p$ to $L^q$, then $\mathcal{T}^+_\alpha$ is bounded from $L^{p,\beta}_+(\omega^p)$ to $L_+^{q,\lambda}(\omega^q)$.\\
{\rm{(ii)}} If $p=1$ and $\mathcal{T}^+_\alpha$ is bounded from $L^1$ to $L^{q,\infty}$, then $\mathcal{T}^+_\alpha$ is bounded from $L^{1,\beta}_{+}(\omega)$ to $WL^{q,\lambda}_{+}(\omega^q)$.
\end{theorem}

Some basic properties of one-sided $A_p$ classes will be introduced. As mentioned above, one-sided $A_p$ weights do not satisfy doubling condition. So, we give the corresponding one-sided doubling condition appeared in \cite{MOD}.
\begin{lemma}\label{Lemma 3.1}\,\
Let $\omega\in A_p^+\; (p\geq1)$. For all $x_0\in \mathbb{R}$ and $h>0$. then
\begin{equation*}\label{(2.1)}
\int _{x_0-h}^{x_0+h}\omega  \lesssim \int _{x_0}^{x_0+h}\omega.
 \end{equation*}
\end{lemma}
In \cite{SF}, Shi and Fu studied the equivalence for one-sided double weight functions.
\begin{lemma}\label{Lemma 7.2}\,\,
Suppose $0<\alpha<1$, $1<p<q<\infty$, and $1/p-1/q=\alpha$. Then the following statements are equivalent:\\
  $\mathrm{(a)}$\,\, $\omega\in A^+_{(p,q)}$.\\
  $\mathrm{(b)}$\,\, $\omega^q \in A_{q(1-\alpha)}^+$.
\end{lemma}
\begin{remark}\label{rem2.1}\,\,
If $\omega\in A^+_{(p,q)}$ with $1\leq p<q<\infty$, we obtain

$$\int _{x_0-h}^{x_0+h}\omega^q \lesssim \int _{x_0}^{x_0+h}\omega^q$$
for all $x_0\in \mathbb{R}$ and $h>0$ by Lemma \ref{Lemma 3.1} and Lemma \ref{Lemma 7.2}.
\end{remark}
\begin{lemma} \label{Lemma 3.2}\,\,
Let $k\in \mathbb{N}$, $x_0\in \mathbb{R}$, $h>0$ and $p, q \geq 1$. Then\\

 $\mathrm{(a)}$\,\, if $\omega \in A_p^+$, we have
 $$\int_{x_0-h-2^kh}^{x_0-h}\omega\lesssim 2^{kp}\int_{x_0-h}^{x_0}\omega.$$

 $\mathrm{(b)}$\,\, if $\omega \in A_{(p,q)}^+$ with $1<p<\infty$, we have
 $$\int_{x_0-h-2^kh}^{x_0-h}\omega^q \lesssim 2^{kq}\int_{x_0-h}^{x_0}\omega^q.$$

 $\mathrm{(c)}$\,\, if $\omega \in A_{(1,q)}^+$, we have
 $$\int_{x_0-h-2^kh}^{x_0-h}\omega^q \lesssim 2^{k}\int_{x_0-h}^{x_0}\omega^q.$$
\end{lemma}

\begin{proof}
For the proof of (b) and the case of $p>1$ in (a), we can refer to \cite[Proposition 12]{SF}. When $p=1$, we claim that
\begin{equation}\label{(2.9)}
|f_{I^-}|\lesssim  \frac{1}{\omega(I^{-}_{-})}\int_{I^-} |f(x)|\omega(x)dx.
 \end{equation}
In fact, for $x\in I^-$, we have
$$\frac{1}{h}\int_{x_0-2h}^{x_0-h}\omega(y)dy\leq \frac{1}{h}\int_{x-2h}^{x}\omega(y)dy.$$
Since $\omega\in A_1^{+}$, by Lemma \ref{Lemma 3.1}, we get
\begin{equation*}\begin{split}
\frac{1}{h}\int_{x-2h}^{x}\omega(y)dy&=\frac{1}{h}\int_{(x-h)-h}^{(x-h)+h}\omega(y)dy
\lesssim \frac{1}{h}\int_{x-h}^{x}\omega(y)dy
\lesssim M^-\omega(x)\lesssim \omega(x).
\end{split}\end{equation*}
Thus
$$\frac{1}{h}\omega(I_{-}^{-})=\frac{1}{h}\int_{x_0-2h}^{x_0-h}\omega(y)dy
\lesssim M^-\omega(x)\lesssim \omega(x).
$$
Consequently,
\begin{align*}
\frac{1}{|I^-|}\int_{I^-}|f(x)|dx=&\frac{1}{|I^-|}\int_{I^-}|f(x)|\omega(x)\omega(x)^{-1}dx \\
\lesssim&\frac{1}{|I^-|}\int_{I^-}|f(x)|\omega(x)\frac{1}{M^-\omega(x)}dx\\
\lesssim& \frac{1}{\omega(I^{-}_{-})}\int_{I^-}|f(x)|\omega(x)dx.
\end{align*}
Replace $I^-$ with $(2^k I)^-$ and then set $f=\chi_{I^-}$, we have
\begin{align*}
\omega((2^k I)^{-}_{-})\lesssim {2^k}\omega(I^-).
\end{align*}

(c) for $\omega \in A_{(1,q)}^+$, we obtain  $\omega^q \in A^+_1 $.
Then by (a), we deduce $\omega^q((2^k I)^{-}_{-}) \lesssim {2^k} \omega^q(I^-).$
\end{proof}

\begin{proof}[Proof of Theorem $\ref{Theorem 1.12}$]
For any $x_0\in\mathbb{R},h>0$ and $\gamma>0$, it suffices to prove
$$\frac{\gamma}{\Phi_{\omega,\lambda,1}^{+}(x_0,h)}\big|\{x\in(x_0,x_0+h):|\mathcal{T}^+f(x)|>\gamma \}\big| \lesssim\|f\|_{L^{1,\lambda}_{+}(\omega)}.$$
Set $f=f_1+f_2 := f\chi_{2I}+f\chi_{({2I})^c}$ to obtain
\begin{align*}
&\frac{\gamma}{\Phi_{\omega,\lambda,1}^{+}(x_0,h)}\big|\{x\in(x_0,x_0+h):|\mathcal{T}^+f(x)|
>\gamma \}\big|\\
&\quad\leq\frac{\gamma}{\Phi_{\omega,\lambda,1}^{+}(x_0,h)}\Big|\Big\{x\in(x_0,x_0+h):|\mathcal{T}^+f_1(x)|
>\frac{\gamma}{2} \Big\}\Big|\\
&\quad\quad+\frac{\gamma}{\Phi_{\omega,\lambda,1}^{+}(x_0,h)}\Big|\Big\{x\in(x_0,x_0+h):
|\mathcal{T}^+f_2(x)|>\frac{\gamma}{2} \Big\}\Big|\\
&\quad=:J_1+J_2.
\end{align*}
Since $\mathcal{T}^+$ is of weak type $(1,1)$, we get
\begin{align*}
J_1  \lesssim \frac{1}{\Phi_{\omega,\lambda,1}^{+}(x_0,h)}\int_{x_0}^{x_0+2h}|f(x)|dx
\lesssim \frac{\omega(x_0-2h,x_0)}{\omega(x_0-h,x_0)}\|f\|_{L^{1,\lambda}_{+}(\omega)}
\lesssim \|f\|_{L^{1,\lambda}_{+}(\omega)},
\end{align*}
where the last inequality follows from Lemma \ref{Lemma 3.1}.
By $(\ref{(1.33)})$, for $x_0<x<x_0+h$, we have
\begin{align*}
|\mathcal{T}^+{f_2}(x)|
\lesssim& \int_{x}^{\infty}\frac{|f_2(y)|}{y-x}dy\\
\lesssim& \sum_{k=1}^\infty\frac{1}{2^{k-1}h}\int_{x_0+2^{k-1}h}^{{x_0+2^{k}h}}|f(y)|dy\\
\lesssim& \|f\|_{L^{1,\lambda}_{+}(\omega)}\sum_{k=1}^{\infty}\frac{{{(2^k}h)}^{\lambda-1}}{{2^k}h}\omega(x_0-h-2^{k+1} h,x_0-h)\\
\lesssim&
\|f\|_{L^{1,\lambda}_{+}(\omega)}\sum_{k=1}^{\infty}\frac{\omega(x_0-h-2^{k+1} h,x_0-h) }{{({2^{k}}h)}^{2-\lambda}}.
\end{align*}
Thus,
  we control $J_2$ by
\begin{align*}
J_2 \lesssim & \frac{1}{\Phi_{\omega,\lambda,1}^{+}(x_0,h)}\int_{x_0}^{x_0+h}|\mathcal{T}^+ {f_2}(x)|dx\\
\lesssim  &\frac{\|f\|_{L^{1,\lambda}_{+}(\omega)}}{{\Phi_{\omega,\lambda,1}^{+}(x_0,h)}}
\int_{x_0}^{x_0+h}\sum_{k=1}^{\infty}\frac{\omega(x_0-h-2^{k+1} h,x_0-h) }{{({2^{k}}h)}^{2-\lambda}}dx\\
\lesssim  & \|f\|_{L^{1,\lambda}_{+}(\omega)}  \sum_{k=1}^{\infty}\frac{1}{2^{k(2-\lambda)}}\frac{\omega(x_0-h-2^{k +1} h,x_0-h)}{\omega(x_0-h,x_0)}\\
\lesssim  &\|f\|_{L^{1,\lambda}_{+}(\omega)}  \sum_{k=1}^{\infty}\frac{1}{2^{k(1-\lambda)}}\\
\lesssim  &\|f\|_{L^{1,\lambda}_{+}(\omega)},
\end{align*}
where (a) of the Lemma \ref{Lemma 3.2} is used in  the fourth inequality. This completes the proof of Theorem \ref{Theorem 1.12}.
\end{proof}
\begin{proof}[Proof of Theorem  $\ref{Theorem 6.12}$] (i)
Let $f=f_1+f_2 := f\chi_{2I}+f\chi_{({2I})^c}$. Then
\begin{align*}
&\left[\frac{1}{\Phi_{\omega,\lambda,q}^{+}(x_0,h)}\int_{x_0}^{x_0+h}|\mathcal{T}^+_\alpha f(x)|^q dx\right]^{1/q}\\
&\quad\leq\left[\frac{1}{\Phi_{\omega,\lambda,q}^{+}(x_0,h)}\int_{x_0}^{x_0+h}|\mathcal{T}^+_\alpha f_1(x)|^q dx\right]^{1/q}\\
&\quad\quad+\left[\frac{1}{\Phi_{\omega,\lambda,q}^{+}(x_0,h)}\int_{x_0}^{x_0+h}|\mathcal{T}^+_\alpha f_2(x)|^q dx\right]^{1/q}\\
&\quad=: K_1+K_2.
\end{align*}
The fact that $\mathcal{T}^+_\alpha$ is  bounded from $L^p$ to $L^q$ implies that
\begin{align*}
K_1 \lesssim& \frac{1}{{\Phi_{\omega,\lambda,q}^{+}(x_0,h)}^{1/q}}\left[\int _{x_0}^{x_0+2h}|f(y)|^pdy \right]^{1/p}
\end{align*}
\begin{align*}
\lesssim&\frac{\|f\|_{L^{p,\beta}_+(\omega^p)}}{h^{1/p-1/q}}\frac{[\omega^p(x_0-2h,x_0)]^{1/p}}
{\left[\omega^q(x_0-h,x_0)\right]^{1/q}}\\
\lesssim&\|f\|_{L^{p,\beta}_+(\omega^p)}\left[\frac{\omega^q(x_0-2h,x_0)}{\omega^q(x_0-h,x_0)}
\right]^{1/q}\\
\lesssim &\|f\|_{L^{p,\beta}_+(\omega^p)},
\end{align*}
where the third inequality follows from H\"{o}lder's inequality and the last inequality is due to Remark \ref{rem2.1}.

Invoking (\ref{(1.35)}), we get
\begin{align*}
|\mathcal{T}^+_\alpha{f_2}(x)|
\lesssim&\int_{x}^{\infty}\frac{|f_2(y)|}{(y-x)^{1-\alpha}}dy\\
\lesssim&\sum_{k=1}^{\infty}\frac{1}{{({2^{k-1}}h)}^{1-\alpha}}\int_{x_0+2^{k-1}h}
^{{x_0+2^{k}h}}|f(y)|dy\\
\lesssim&
\sum_{k=1}^{\infty}\frac{1}{{({2^k}h)}^{1-\alpha}}\left[\int_{x_0-h}^{{x_0-h+2^{k+1}h}}
|f(y)|^pdy\right]^{1/p}{(2^k h)^{1-1/p}}\\
\lesssim&
\sum_{k=1}^{\infty}\frac{\left[\omega^p(x_0-h-2^{k+1}h,x_0-h)\right]^{1/p}}{{({2^k}h)}
^{1/p-\alpha-(\beta-1)/p}}\\
&\quad\times\left[\frac{(2^{k+1}h)^{1-\beta}}{\omega^p(x_0-h-2^{k+1}h,x_0-h)}\int _{x_0-h}^{x_0-h+2^{k+1}h}|f(y)|^p dy\right]^{1/p}\\
\lesssim&\|f\|_{L^{p,\beta}_+(\omega^p)}\sum_{k=1}^{\infty}\frac{\left[\omega^p
(x_0-h-2^{k+1}h,x_0-h)\right]^{1/p}}{{({2^k} h)}^{1/p-\alpha-(\beta-1)/p}}.
\end{align*}
Substituting the above estimate into $K_2$, we obtain
\begin{align*}
K_2\lesssim&
\left\{\frac{\|f\|^{q}_{L^{p,\beta}_{+}(\omega^p)}}{{\Phi_{\omega,\lambda,q}^{+}(x_0,h)}} \int_{x_0}^{x_0+h}\bigg \{\sum_{k=1}^{\infty}
\frac{\left[\omega^p(x_0-h-2^{k+1}h,x_0-h)\right]^{1/p}}{{({2^k} h)}^{1/p-\alpha-(\beta-1)/p}}\bigg\}^qdx\right\}^{1/q}
\\
\lesssim& \|f\|_{L^{p,\beta}_+(\omega^p)} h^{(2-\lambda)/q} \sum_{k=1}^\infty
({2^k}h)^{\alpha-\frac{1}{p}+\frac{\beta-1}{p}}\frac{\left[\omega^p ({x_0-h-2^{k+1}h},{x_0-h})\right]^{1/p}}{\left[\omega^q (x_0-h,x_0)\right]^{1/q}}\\
\lesssim&
 \|f\|_{L^{p,\beta}_+(\omega^p)} \sum_{k=1}^\infty 2^{k(\alpha-\frac{1}{q}+\frac{\beta-1}{p})}\left[\frac{{\omega^q(x_0-h-2^{k+1}h,x_0-h)}}
{{\omega^q(x_0-h,x_0)}}\right]^{1/q}\\
\lesssim&
\|f\|_{L^{p,\beta}_+(\omega^p)} \sum_{k=1}^\infty 2^{k(\alpha-\frac{1}{q}+\frac{\beta-1}{p}+1)}\\
\lesssim&
\|f\|_{L^{p,\beta}_+(\omega^p)},
\end{align*}
where the fourth inequality follows from H\"{o}lder's inequality, and (b) of the Lemma \ref{Lemma 3.2} is used in the last inequality.

From the above discussion, we conclude that
$$\|\mathcal{T}^+_\alpha{f}\|_{L^{q,\lambda}_+(\omega^q)}
\lesssim\|f\|_{L^{p,\beta}_+(\omega^p)}.$$

Now we turn to prove (ii). The proof is similar to that of Theorem \ref{Theorem 1.12}. We write
\begin{align*}
&\bigg[\frac{\gamma^{q}}{\Phi_{\omega,\lambda,q}^{+}(x_0,h)}\big|
\{x\in(x_0,x_0+h):|\mathcal{T}^+_\alpha f(x)|>\gamma \}\big|\bigg]^{1/q}\\
&\quad\leq \bigg[\frac{\gamma^{q}}{\Phi_{\omega,\lambda,q}^{+}(x_0,h)}\big|
\{x\in(x_0,x_0+h):|\mathcal{T}^+_\alpha f_1(x)|>\frac{\gamma}{2} \}\big|\bigg]^{1/q}\\
&\quad\quad +\bigg[\frac{\gamma^{q}}{\Phi_{\omega,\lambda,q}^{+}(x_0,h)}\big|
\{x\in(x_0,x_0+h):|\mathcal{T}^+_\alpha f_2(x)|>\frac{\gamma}{2} \}\big|\bigg]^{1/q}\\
&\quad=:L_1+L_2.
\end{align*}
For $L_1$, it follows from Remark \ref{rem2.1} and the fact $\mathcal{T}^+_\alpha$ is  bounded from $L^1$ to $L^{q,\infty}$ that
\begin{align*}
L_1 \lesssim& \frac{1}{\Phi_{\omega,\lambda,q}^{+}(x_0,h)^{1/q}}\int _{x_0}^{x_0+2h}|f(y)|dy\\
\lesssim&\|f\|_{L^{1,\beta}_{+}(\omega)}h^{\beta-1-\frac{\lambda-1}{q}}\frac{\omega(x_0-2h,x_0)}
{\left[\omega^q(x_0-h,x_0)\right]^{1/q}}\\
\lesssim&\|f\|_{L^{1,\beta}_{+}(\omega)}\left[\frac{\omega^q(x_0-2h,x_0)}{\omega^q(x_0-h,x_0)}
\right]^{1/q}\\
\lesssim &\|f\|_{L^{1,\beta}_{+}(\omega)}.
\end{align*}
By $(\ref{(1.35)})$, we get
\begin{align*}
|\mathcal{T}^+_\alpha{f_2}(x)|
\lesssim&\sum_{k=1}^{\infty}\frac{1}{{({2^k}h)}^{1-\alpha}}\int_{x_0+2^{k-1}h}
^{{x_0+2^{k}h}}|f(y)|dy\\
\lesssim& \|f\|_{L^{1,\beta}_{+}(\omega)}\sum_{k=1}^{\infty}\frac{\omega(x_0-h-2^{k+1} h,x_0-h)}{{({2^k} h)}^{2-\alpha-\beta}}.
\end{align*}
Thus
\begin{align*}
L_2
\lesssim&\left\{\frac{\|f\|^{q}_{L^{1,\beta}_{+}(\omega)}}{{\Phi_{\omega,\lambda,q}^{+}(x_0,h)}}
\int_{x_0}^{x_0+h}\bigg[\sum_{k=1}^\infty
\frac{\omega(x_0-h-2^{k+1}h,x_0-h)}{({2^k}h)^{2-\alpha-\beta}}\bigg]^q dx\right\}^{1/q}
\end{align*}
\begin{align*}
\lesssim& \|f\|_{L^{1,\beta}_{+}(\omega)}h^{\frac{1}{q}-1}
\sum_{k=1}^\infty2^{k(\alpha+\beta-2)}\frac{\omega(x_0-h-2^{k+1}h,x_0-h)}
{\left[\omega^q(x_0-h,x_0)\right]^{1/q}}\\
\lesssim&\|f\|_{L^{1,\beta}_{+}(\omega)}\sum_{k=1}^\infty 2^{k(\alpha+\beta-1-\frac{1}{q})}\left[\frac{{\omega^q(x_0-h-2^{k+1}h,x_0-h)}}
{{\omega^q(x_0-h,x_0)}}\right]^{1/q}\\
\lesssim&\|f\|_{L^{1,\beta}_{+}(\omega)} \sum_{k=1}^\infty{2^{k(\alpha+\beta-\frac{1}{q})}}\\
\lesssim&\|f\|_{L^{1,\beta}_{+}(\omega)},
\end{align*}
where the last inequality derives from (c) of Lemma \ref{Lemma 3.2}. Consequently,
$$\|\mathcal{T}^+_\alpha{f}\|_{WL^{q,\lambda}_+(\omega^q)}
\lesssim \|f\|_{L^{1,\beta}_+(\omega)}.$$
This completes the proof of Theorem \ref{Theorem 6.12}.
\end{proof}

\section{Locally one-sided weighted Morrey space and its application }

The subject of fractional derivatives and integrals has gained considerable popularity and importance during the past several decades,
due mainly to its demostrated applications in numerous seemingly diverse and widespread fields of science and engineering. It does indeed
provide several potentially useful tools for solving differential and integral equations, and various other problems. In recent years
many articles and books on this subject have been
published (see \cite{KST} and references therein). In this section, we only consider the situation that the fractional index $\alpha\in(0,1)$.
\begin{definition}[\cite{KST}] For $0<\alpha<1$,
the Riemann-Liouville integral of order $\alpha$ is defined as
$$I_{0^+}^{\alpha}f(t)=\frac{1}{\Gamma(\alpha)}\int_{0}^{t}\frac{f(\tau)}{(t-\tau)^{1-\alpha}}
d\tau.$$
The Riemann-Liouville derivative of order $\alpha$ is defined by
$$D_{0^+}^{\alpha}f(t)=\frac{1}{\Gamma(1-\alpha)}\frac{d}{dt}\int_{0}^{t}\frac{f(\tau)}
{(t-\tau)^{1-\alpha}}d\tau.$$
Here $0<t<T$, for some $T>0$. It is clear that $D_{0^+}^{\alpha}f(t)=\frac{d}{dt}I_{0+}^{1-\alpha}f(t)$.
\end{definition}
Many authors studied the representation of functions by fractional integrals in continuous or Lebesgue spaces (see \cite{SKM}). Now, we consider the representation of functions by fractional integrals in local one-sided weighted Morrey space. For convenience, we work on a local one-sided weighted Morrey space $L^{p,\lambda}_+((0,T),\omega^\theta)$, which is defined by
$$\|f\|_{{L^{p,\lambda}_{+}((0,T),\omega^\theta)}}:=\sup_{\substack{0\leq x_0\leq T\\ 0<h\leq T}}\Big(\frac{1}{\Phi_{\omega,\lambda,\theta}^{+}(x_0,h)}\int_{(x_0,x_0+h)\bigcap(0,T)}
|f(y)|^pdy\Big)^{1/p}<\infty$$
for some $T>0$, where ${\Phi_{\omega,\lambda,\theta}^{+}(x_0,h)}$ is given in (\ref{eq2.1}).
%
The following theorem gives the boundedness of Riemann-Liouville integral operators on local one-sided weighted Morrey spaces.

\begin{theorem}\label{Theorem 4.1}\,
Let $\omega\in A_{(p,q)}^+$, and let $0<\sigma<1$, $1<p,q<\infty$, ${1/p}<\alpha<1$, $0\leq\beta, \mu<1$, $1/p-1/q=\sigma$ and $q+\mu\leq2$. Then
$I_{0^{+}}^{\alpha}$ is bounded from $L^{p,\beta}_+((0,T),\omega^p)$ to $L^{q,\mu}_+((0,T),\omega^q)$.
\end{theorem}

\begin{proof}
For any $0\leq x_0\leq T$ and $0<h\leq T$, using H\"older's inequality, we have
\begin{align*}
&\frac{1}{\Phi_{\mu,q}^{\omega,+}(x_0,h)}\int_{(x_0,x_0+h)\cap(0,T)}|I_{0^{+}}^{\alpha}f(t)|^qdt\\
&\quad\lesssim\frac{1}{\Phi_{\mu,q}^{\omega,+}(x_0,h)}\int_{(x_0,x_0+h)\cap(0,T)}
\left[\int_{0}^{t}|f(\tau)|^pd\tau\right]^{q/p}\left[\int_{0}^{t}{(t-\tau)^{(\alpha-1) p'}}d\tau\right]^{{q/p'}}dt\\
&\quad\lesssim\frac{1}{\Phi_{\mu,q}^{\omega,+}(x_0,h)}\int_{(x_0,x_0+h)\cap(0,T)}
\left[\int_{0}^{t}|f(\tau)|^pd\tau\right]^{{q/p}}t^{\frac{(\alpha p-1)q}{p}}dt
\\&\quad=:K.
\end{align*}
For $0<h\leq T$, there exists a nonnegative integer $k$ such that $T/2^{k+1}<h\leq T/2^{k}$, then
\begin{align*}
K\lesssim&
\frac{1}{\Phi_{\mu,q}^{\omega,+}(x_0,h)}\int_{(x_0,x_0+h)\cap(0,T)}
\bigg[\frac{T^{\beta-1}\omega^p(-T,0)}{T^{\beta-1}\omega^p(-T,0)} \int_{0}^T|f(\tau)|^pd\tau\bigg]^{{q/p}}t^{\frac{(\alpha p-1)q}{p}}dt
\\
\lesssim& \|f\|_{L^{p,\beta}_+((0,T),\omega^p)}^q\frac{T^{q(\beta-1)/p}\left[\omega^p(-T,0)\right]^{q/p}}{{h^{\mu-1}}
\omega^q(x_0-h,x_0)}\int_{(x_0,x_0+h)\cap(0,T)}t^{\frac{(\alpha p-1)q}{p}}dt
\\
\lesssim& \|f\|_{L^{p,\beta}_+((0,T),\omega^p)}^q\frac{T^{q\beta/p-1}\omega^q(-T,0)}{{h^{\mu-1}}
\omega^q(x_0-h,x_0)}\int_{(x_0,x_0+h)\cap(0,T)}t^{\frac{(\alpha p-1)q}{p}}dt\\
\lesssim&
\|f\|_{L^{p,\beta}_+((0,T),\omega^p)}^q\frac{T^{\beta q/p-1}}{h^{\mu-2}}\frac{\omega^q(x_0-2T,x_0)}{\omega^q(x_0-h,x_0)}
T^{\frac{(\alpha p-1)q}{p}}\\
\lesssim&
\|f\|_{L^{p,\beta}_+((0,T),\omega^p)}^qT^{1-\mu+\frac{(\alpha p+\beta -1)q}{p}}2^{k\mu-2k}\frac{\omega^q(x_0-2^{k+2}h,x_0)}{\omega^q(x_0-h,x_0)}\\
\lesssim&
\|f\|_{L^{p,\beta}_+((0,T),\omega^p)}^qT^{1-\mu+\frac{(\alpha p+\beta-1)q}{p}}2^{k\mu-2k}2^{(k+3)q}
\\
\lesssim&
\|f\|_{L^{p,\beta}_+((0,T),\omega^p)}^qT^{1-\mu+\frac{(\alpha p+\beta-1)q}{p}},
\end{align*}
where the sixth inequality follows from (b) of Lemma \ref{Lemma 3.2} and $q+\mu\leq2$ are used in the last inequality. Consequently,
$$\|I_{0^{+}}^{\alpha}f\|_{L^{q,\mu}_+((0,T),\omega^q)}
\lesssim\|f\|_{L^{p,\beta}_+((0,T),\omega^p)},$$
which completes the proof of Theorem \ref{Theorem 4.1}.
\end{proof}
To study the compactness of operators on $L^{p,\beta}_+((0,T),\omega)$, we first establish a weighted Fr\'{e}chet-Kolmogorov theorem.

\begin{theorem}\label{Theorem 4.2}\,
Suppose $1< p<\infty$, $0\leq\beta<1$, $p+\beta\leq2$ and $\omega\in A_p^+$. Let $G$ be a subset of $L^{p,\beta}_+((0,T),\omega)$. Then $G$ is strongly pre-compact set in $L^{p,\beta}_+((0,T),\omega)$ if it satisfies the following conditions:
\begin{enumerate}
\item G is uniformly bounded, i.e. $\sup\limits_{f\in G}\|f\|_{L^{p,\beta}_+((0,T),\omega)}<\infty$;
\item G is uniformly equicontinuous, i.e. $\sup\limits_{f\in G}\|f(\cdot+l)-f(\cdot)\|_{L^{p,\beta}_+((0,T),\omega)}\rightarrow0$, \emph as $l \rightarrow 0$;
\item G is uniformly vanishes at infinity, i.e. $\sup\limits_{f\in G}\|f\chi_{E_{R}}\|_{L^{p,\beta}_+((0,T),\omega)}\rightarrow0$, \emph as $R \rightarrow T$, where $E_{R}:=\{x\in(0,T): x>R\}$.
\end{enumerate}

\end{theorem}

\begin{proof}
For $t>0$, denote the mean value of $f$ in $G$ by
$$M_tf(x):=\frac{1}{t}\int_0^tf(x+y)dy=\frac{1}{t}\int_x^{x+t}f(y)dy.$$
By H\"{o}lder's inequality,  we get
\begin{align*}
&\bigg[\frac{1}{\Phi_{\beta,1}^{\omega,+}(x_0,h)}\int_{(x_0,x_0+h)\cap(0,T)}|M_tf(x)- f(x)|^pdx\bigg]^{\frac{1}{p}}\\
&\quad\leq\bigg[\frac{1}{\Phi_{\beta,1}^{\omega,+}(x_0,h)}\int_{(x_0,x_0+h)\cap(0,T)}\frac{1}{t}\int_0^t
|f(x+y)- f(x)|^pdydx\bigg]^{\frac{1}{p}}
\end{align*}
\begin{align*}
&\quad\leq\bigg[\frac{1}{t}\int_0^t\frac{1}{\Phi_{\beta,1}^{\omega,+}(x_0,h)}\int_{(x_0,x_0+h)\cap(0,T)}
|f(x+y)- f(x)|^pdxdy\bigg]^{\frac{1}{p}}\\
&\quad\leq\sup_{0<|y|\leq t}\|f(\cdot+y)-f(\cdot)\|_{L^{p,\beta}_+((0,T),\omega)} .
\end{align*}
Hence
\begin{align}\label{eq4.1}
\lim_{t\rightarrow0}\|M_tf-f\|_{L^{p,\beta}_+((0,T),\omega)}\leq\sup_{0<|y|\leq t}\|f(\cdot+y)-f(\cdot)\|_{L^{p,\beta}_+((0,T),\omega)}.
\end{align}
This together with condition $(2)$, $(3)$ in Theorem \ref{Theorem 4.2}, we have
\begin{align}\label{4.2}
\lim_{t\rightarrow0}\|M_tf-f\|_{L^{p,\beta}_+((0,T),\omega)}=0, \quad \text{uniformly in} \:f\in G,
\end{align}
and $\{M_tf: f\in G\}\subset L^{p,\beta}_+((0,T),\omega)$ satisfies $\sup_{f\in G}\|M_tf\|_{L^{p,\beta}_+((0,T),\omega)}<\infty$.
According to condition $(3)$ on $G$, for $\varepsilon>0$, there exists a constant $0<R<T$ such that
\begin{align}\label{4.1}
 \sup\limits_{f\in G}\|f\chi_{E_{R}}\|_{L^{p,\beta}_+((0,T),\omega)}<\frac{\varepsilon}{8}.
\end{align}
Now we prove that for each fixed $t$, which is sufficiently, the set $\{M_t f: f\in G \}$ is a strongly pre-compact set in $C([0,R])$, where $C([0,R])$ denotes the continuous function space on $[0,R]$ with uniform norm. By the Ascoli-Arzel\`{a }theorem, we only need to show that $\{M_t f: f\in G \}$ is bounded and equicontinuous on $C([0,R])$.
For $f\in G$, $0\leq x\leq R$ and $x+t\leq T$, we get
\begin{align*}
|M_tf(x)|\leq&\frac{1}{t}\int_x^{x+t}|f(y)|dy\\
\leq&\Big[\frac{1}{t}\int_x^{x+t}| f(y)|^pdx\Big]^{\frac{1}{p}}\\
\leq&\Big[\frac{{t^{\beta-2}}\omega(x-t,x)}{{t^{\beta-1}}\omega(x-t,x)}\int_{(x,x+t)\bigcap(0,T)}
|f(y)|^pdy\Big]^{\frac{1}{p}}\\
\leq&\Big[{t^{\beta-2}}\omega(-R,R)\Big]^{\frac{1}{p}}\|f\|_{L^{p,\beta}_+((0,T),\omega)}.
\end{align*}

On the other hand, for any $x_1,x_2\in [0,R]$ and $x_2+t\leq T$, we obtain
\begin{align*}
&|M_tf(x_1)-M_tf(x_2)|\\
\;\quad&=\Big|\frac{1}{t}\int_{x_1}^{x_1+t}f(y)dy-\frac{1}{t}
\int_{x_2}^{x_2+t}f(y)dy\Big|\\
\;\quad&=\Big|\frac{1}{t}\int_{x_2}^{x_2+t}f(y+x_2-x_1)dy-\frac{1}{t}
\int_{x_2}^{x_2+t}f(y)dy\Big|\\
\;\quad&\leq\Big]\frac{1}{t}\int_{x_2}^{x_2+t}|f(y+x_2-x_1)-f(y)|^pdy\Big]^{\frac{1}{p}}\\
\;\quad&\leq\Big[\frac{{t^{\beta-2}}\omega(x_2-t,x_2)}{{t^{\beta-1}}\omega(x_2-t,x_2)}
\int_{(x_2,x_2+t)\bigcap(0,T)}|f(y+x_2-x_1)-f(y)|^pdy\Big]^{\frac{1}{p}}\\
\;\quad&\leq\Big[{t^{\beta-2}}\omega(-R,R)\Big]^{\frac{1}{p}}
\|f(\cdot+x_2-x_1)-f(\cdot)\|_{L^{p,\beta}_+((0,T),\omega)}.
\end{align*}
This together with condition $(2)$ implies the equicontinuity of $\{M_t f: f\in G \}$.

Next we turn to prove that for small enough $t$, the set $\{M_t f: f\in G \}$ is also a strongly pre-compact set in ${L^{p,\beta}_+((0,T),\omega)}$. Because the set $\{M_t f: f\in G \}$ is totally bounded set in $C([0,R])$. For the above $\v$, there exist a finite number of $\{f_1,f_2,\ldots,f_m\}\subset G$ such that
\begin{align}\label{eq4.5}
\sup_{0\leq x\leq R}|M_tf(x)- M_tf_j(x)|<\v\Big[\omega(-R,0)\Big]^{\frac{1}{p}}.
\end{align}
We only need to show that for any $f\in G$, $0\leq x_0\leq T$ and $0<h\leq T$, there exist $f_j$$(1\leq j\leq m)$ such that
\begin{align}\label{4.4}
H_1:=\bigg[\frac{1}{\Phi_{\beta,1}^{\omega,+}(x_0,h)}\int_{(x_0,x_0+h)\cap(0,T)}|M_tf(x)- M_tf_j(x)|^pdx\bigg]^{\frac{1}{p}}\lesssim\v.
\end{align}
Case 1: For $0\leq x_0\leq R$ and $x_0+h\leq R$, there exists a nonnegative integer $k$ such that $R/2^{k+1}< h\leq R/2^{k}$. By (\ref{eq4.5}) and (a) of Lemma \ref{Lemma 3.2}, we obtain
\begin{align*}
H_1\leq&\v\bigg[\frac{1}{\Phi_{\beta,1}^{\omega,+}(x_0,h)}\int_{(x_0,x_0+h)\cap[0,R]}
\omega(-R,0)dx\bigg]^{\frac{1}{p}}\\
\leq&\v\Big[\frac{h^{2-\beta}\omega(-R,0)}{\omega(x_0-h,x_0)}\Big]^{\frac{1}{p}}\\
\leq&\v\Big[\frac{2^{k(\beta-2)}R^{2-\beta}\omega(x_0-2^{k+2}h,x_0)}
{\omega(x_0-h,x_0)}\Big]^{\frac{1}{p}}\\
\lesssim&\v2^{\frac{k(p+\beta-2)}{p}}T^{\frac{2-\beta}{p}}
\lesssim\v\nonumber.
\end{align*}
Case 2: For $R<x_0\leq T$, by (\ref{4.2}) and (\ref{4.1}), we have
\begin{align*}
H_1=&\bigg[\frac{1}{\Phi_{\beta,1}^{\omega,+}(x_0,h)}\int_{(x_0,x_0+h)\cap(R,T)}|M_tf(x)- M_tf_j(x)|^pdx\bigg]^{\frac{1}{p}}\\
\leq&\bigg[\frac{1}{\Phi_{\beta,1}^{\omega,+}(x_0,h)}\int_{(x_0,x_0+h)\cap(R,T)}|M_tf(x)- f(x)|^pdx\bigg]^{\frac{1}{p}}\\
&+\bigg[\frac{1}{\Phi_{\beta,1}^{\omega,+}(x_0,h)}\int_{(x_0,x_0+h)\cap(R,T)}|f(x)- f_j(x)|^pdx\bigg]^{\frac{1}{p}}\\
&+\bigg[\frac{1}{\Phi_{\beta,1}^{\omega,+}(x_0,h)}\int_{(x_0,x_0+h)\cap(R,T)}
|M_tf_j(x)- f_j(x)|^pdx\bigg]^{\frac{1}{p}}\\
\leq&\|M_tf-f\|_{L^{p,\beta}_+((0,T),\omega)}+\|M_tf_j-f_j\|_{L^{p,\beta}_+
((0,T),\omega)}\\
&+\|f_j\chi_{E_{R}}\|_{L^{p,\beta}_+((0,T),\omega)}
+\|f\chi_{E_{R}}\|_{L^{p,\beta}_+((0,T),\omega)}\\
\lesssim&\v.
\end{align*}
Case 3: For $0\leq x_0\leq R$ and $x_0+h>R$, the proof is quite similar to that of Case 1 and Case 2. In fact,
\begin{align*}
H_1\leq&\bigg[\frac{1}{\Phi_{\beta,1}^{\omega,+}(x_0,h)}\int_{(x_0,x_0+h)\cap[0,R]}|M_tf(x)- M_tf_j(x)|^pdx\bigg]^{\frac{1}{p}}
\end{align*}
\begin{align*}
&+
\bigg[\frac{1}{\Phi_{\beta,1}^{\omega,+}(x_0,h)}\int_{(x_0,x_0+h)\cap(R,T)}|M_tf(x)- M_tf_j(x)|^pdx\bigg]^{\frac{1}{p}}\\
=:&H_{11}+H_{12}.
\end{align*}
If $h\leq R$, then the conclusion (\ref{4.4}) in this case can be deduced from Case 1 and Case 2. If $R< h\leq T$, we get $H_{12}\lesssim \v$ by the same discussion in Case 2. For $H_1$, we choose a nonnegative integer $k$ such that $T/2^{k+1}< h\leq T/2^{k}$. Proceeding the proof as we do in Case (1), we have $H_{11}\lesssim \v$.

Finally, we show that the set $G$ is a relative compact set in ${L^{p,\beta}_+((0,T),\omega)}$. Take any sequence $\{f_j\}_{j=1}^\infty\subset G$. Since the relative compactness of $\{M_t f: f\in G \}\subset {L^{p,\beta}_+((0,T),\omega)}$, there exists a subsequence $\{M_tf_{j_k}\}_{k=1}^\infty$ of $\{M_tf_j: f_j\}$ that is convergent in ${L^{p,\beta}_+((0,T),\omega)}$. For any $\v>0$ there exists $N\in\mathbb{N}$ such that for any $k>N$ and $m\in \mathbb{N}$, $\|M_tf_{j_k}-M_tf_{j+m}\|_{L^{p,\beta}_+((0,T),\omega)}<\v$. For any $0<h\leq T$ and $x_0\geq0$, using (\ref{4.2}), we get
\begin{align*}
&\bigg[\frac{1}{\Phi_{\beta,1}^{\omega,+}(x_0,h)}\int_{(x_0,x_0+h)\cap(0,T)}|f_j(x)- f_{j+m}(x)|^pdx\bigg]^{\frac{1}{p}}\\
&\quad\leq\bigg[\frac{1}{\Phi_{\beta,1}^{\omega,+}(x_0,h)}\int_{(x_0,x_0+h)\cap(0,T)}
|f_j(x)- M_tf_j(x)|^pdx\bigg]^{\frac{1}{p}}\\
&\quad\quad+\bigg[\frac{1}{\Phi_{\beta,1}^{\omega,+}(x_0,h)}\int_{(x_0,x_0+h)\cap(0,T)}
|M_tf_{j+m}(x)- M_tf_j(x)|^pdx\bigg]^{\frac{1}{p}}\\
&\quad\quad+\bigg[\frac{1}{\Phi_{\beta,1}^{\omega,+}(x_0,h)}\int_{(x_0,x_0+h)\cap(0,T)}
|M_tf_{j+m}(x)- f_{j+m}(x)|^pdx\bigg]^{\frac{1}{p}}\\
&\quad\leq\|M_tf_{j}-f_{j}\|_{L^{p,\beta}_+((0,T),\omega)}+\|M_tf_{j}-M_tf_{j+m}\|
_{L^{p,\beta}_+((0,T),\omega)}\\
&\quad\quad+\|M_tf_{j+m}-f_{j+m}\|_{L^{p,\beta}_+((0,T),\omega)}\\
&\quad\lesssim \v.
\end{align*}
This implies that the subsequence $\{f_j\}_{j=1}^\infty$ converges in ${L^{p,\beta}_+((0,T),\omega)}$. By noting that ${L^{p,\beta}_+((0,T),\omega)}$ is a Banach space. Consequently, the set $G$ is a relative compact set in ${L^{p,\beta}_+((0,T),\omega)}$.
\end{proof}
\begin{proposition}\label{prop1}
If $0<\sigma<1$, $1<p, q<\infty$, $\frac{1}{2}(1+\frac{1}{p})<\alpha<1$, $\omega \in A^+_{(p,q)}$, $0\leq\beta,\mu<1$, $1/p-1/q=\sigma$ and $q+\mu\leq2$, then
$$Tu(t)=\frac{\chi_{(0,T)}(t)}{\Gamma(\alpha)}\int_{0}^{t}\frac{u(\tau)}{(t-\tau)^{1-\alpha}}d\tau,$$
is a compact operator from ${L}^{p,\beta}_+((0,T),\omega^p)$ to ${L}^{q,\mu}_+((0,T),\omega^q)$.
\end{proposition}
\begin{proof}
Suppose that $\mathbb{K}$ is an arbitrary bounded set in ${L}^{p,\beta}_+((0,T),\omega^p)$, it suffices to prove that $T(\mathbb{K})$ is strongly pre-compact set in ${L}^{q,\mu}_+((0,T),\omega^q)$. By Theorem \ref{Theorem 4.2}, we only need to
prove that $(1)-(3)$ hold uniformly in $T(\mathbb{K})$.

$(1)$.\,\, Without loss of generality, for any $u\in \mathbb{K}$, we assume $\|u\|_{{L}^{p,\beta}_+((0,T),\omega^p)}\leq M$ for some $M>0$.
Note that $Tu(t)=\chi_{(0,T)}(t)I_{0^{+}}^{\alpha}f(t)$. Then
\begin{align*}\begin{split}
\frac{1}{\Phi_{\mu,q}^{\omega,+}(x_0,h)}\int_{(x_0,x_0+h)\cap(0,T)}|Tu(t)|^qdt
=
\frac{1}{\Phi_{\mu,q}^{\omega,+}(x_0,h)}\int_{(x_0,x_0+h)\cap(0,T)}\left|I_{0^{+}}^{\alpha}f(t)
\right|^qdt.
\end{split}
\end{align*}
By the same discussion as in Theorem \ref{Theorem 4.1}, we get
$$\|Tu\|_{L^{q,\mu}_+((0,T),\omega^q)}\lesssim M.$$
$(2)$.\,\, For any $x\geq T,~y>0$, it is easy to see that
\begin{equation*}
|Tu(x+y)-Tu(x)|=0.\label{0}
\end{equation*}
For $x\in (0,T)$ and $y>0$ small enough such that $x+y\in (0,T)$. Then
\begin{align*}
|Tu(x+y)-Tu(x)|
=&\frac{1}{\Gamma(\alpha)}\Big|\int_{0}^{x+y}\frac{u(\tau)}{(x+y-\tau)^{1-\alpha}}d\tau-
\int_{0}^{x}\frac{u(\tau)}{(x-\tau)^{1-\alpha}}d\tau\Big|
\\
\leq&\frac{1}{\Gamma(\alpha)}\Big|\int_{0}^{x}\frac{u(\tau)}{(x+y-\tau)^{1-\alpha}}
-\frac{u(\tau)}{(x-\tau)^{1-\alpha}}d\tau
\Big|
\\
&+\frac{1}{\Gamma(\alpha)}\Big|
\int_{x}^{x+y}\frac{u(\tau)}{(x+y-\tau)^{1-\alpha}}d\tau
\Big|\\
=&:F_1(x)+F_2(x).
\end{align*}
Recall that $|x^{\lambda}-y^{\lambda}|\leq |x-y|^{\lambda}$, $x,y\geq 0$, $0<\lambda<1$.
The fact that $\frac{1}{2}(1+\frac{1}{p})<\alpha<1$ and  H\"older's inequality imply that
\begin{equation*}\begin{split}
F_1(x)&=\frac{1}{\Gamma(\alpha)}\Big|\int_{0}^{x}u(\tau)\frac{(x-\tau)^{1-\alpha}-(x+y-\tau)^{1-\alpha}}
{(x+y-\tau)^{1-\alpha}(x-\tau)^{1-\alpha}}d\tau
\Big|\\
&\lesssim\int_{0}^{x}\frac{|u(\tau)|y^{1-\alpha}}
{(x+y-\tau)^{1-\alpha}(x-\tau)^{1-\alpha}}d\tau\\
&\lesssim y^{1-\alpha}\left[\int_{0}^{x}|u(\tau)|^pd\tau\right]^{1/p}\left[\int_{0}^{x}\frac{1}
{(x+y-\tau)^{(1-\alpha)p'}(x-\tau)^{(1-\alpha)p'}}d\tau\right]^{1/p'}\\
&\lesssim y^{1-\alpha}\left[\int_{0}^{x}|u(\tau)|^pd\tau\right]^{1/p}
x^{\frac{2\alpha p-1-p}{p}}.
\end{split}\end{equation*}
For $F_2(x)$, by H\"older's inequality again, we have
\begin{equation*}\begin{split}
F_2(x)&\leq\frac{1}{\Gamma(\alpha)}\left[
\int_{x}^{x+y}|u(\tau)|^pd\tau\right]^{1/p}\left[
\int_{x}^{x+y}\frac{1}{(x+y-\tau)^{(1-\alpha) p'}}d\tau\right]^{1/p'}\\
&\lesssim \left[
\int_{x}^{x+y}|u(\tau)|^pd\tau\right]^{1/p}
y^{\frac{\alpha p-1}{p}}.\label{2+}
\end{split}\end{equation*}
If $x_0\geq T$ and any $h>0$, we get
$$\frac{1}{\Phi_{\mu,q}^{\omega,+}(x_0,h)}\int_{(x_0,x_0+h)\cap(0,T)}|Tu(x+y)-Tu(x)|^qdx
=0.$$
If $0\leq x_0<T$, combining with estimates for $F_1(x)$ and $F_2(x)$ lead to
\begin{align*}
&\frac{1}{\Phi_{\mu,q}^{\omega,+}(x_0,h)}\int_{(x_0,x_0+h)\cap(0,T)}
|Tu(x+y)-Tu(x)|^qdx\\
&\quad\leq
\frac{1}{\Phi_{\mu,q}^{\omega,+}(x_0,h)}\int_{(x_0,x_0+h)\cap(0,T)}|F_1(x)+F_2(x)|^qdx
\\
&\quad\lesssim {\frac{y^{(1-\alpha)q}}{\Phi_{\mu,q}^{\omega,+}(x_0,h)}}\int_{(x_0,x_0+h)\cap(0,T)}
\left[\int_{0}^{x}|u(\tau)|^pd\tau\right]^{q/p}
x^{\frac{(2\alpha p-1-p)q}{p}}dx\\
&\quad\quad+\frac{y^{{(\alpha p-1)q/p}}}{{\Phi_{\mu,q}^{\omega,+}(x_0,h)}}\int_{(x_0,x_0+h)\cap(0,T)}\left[
\int_{x}^{x+y}|u(\tau)|^pd\tau
\right]^{q/p}dx\\
&\quad=:F_{11}+F_{21}.
\end{align*}
By the similar discussion as in Theorem \ref{Theorem 4.1}, we have
$$F_{11}\lesssim y^{(1-\alpha)q}\|u\|_{{L}^{p,\beta}_+((0,T),\omega^p)}^q
\lesssim y^{(1-\alpha)q} M^q.$$
In view of the estimate for $K$ in Theorem \ref{Theorem 4.1}, we obtain
\begin{equation*}\begin{split}
F_{21}
\lesssim&\frac{y^{{(\alpha p-1)q/p}}}{{{\Phi_{\mu,q}^{\omega,+}(x_0,h)}}}\int_{(x_0,x_0+h)\cap(0,T)}\Big[
\int_{0}^{T}|u(\tau)|^pd\tau
\Big]^{q/p}dx
\lesssim y^{{(\alpha p-1)q/p}}M^q.
\end{split}\end{equation*}
Thus
$$\|Tu(\cdot+y)-Tu(\cdot)\|_{L^{q,\mu}_+((0,T),\omega^q)}\lesssim M\left(y^{1-\alpha}+y^{\frac{\alpha p-1}{p}}\right).$$
Recalling from $(1+1/p)/2<\alpha<1$, we get $1-\alpha>0$ and $(\alpha p-1)/p>0$. Then
$$\lim_{y\rightarrow 0}\|Tu(\cdot+y)-Tu(\cdot)\|_{L^{q,\mu}_+((0,T),\omega^q)}
=0$$
uniformly in $Tu\in T(\mathbb{K})$.

$(3)$. It is easy to check that for any $u\in \mathbb{K}$ and $R>T$
\begin{align*}
\|Tu\chi_{E_{R}}\|_{L^{q,\mu}_+((0,T),\omega^q)}^q=&
\sup_{\substack{0\leq x_0\leq T\\ 0<h\leq T\\}}\frac{1}{\Phi_{\mu,q}^{\omega,+}(x_0,h)}\int_{(x_0,x_0+h)\bigcap(0,T)}
|Tu\chi_{E_{R}}(t)|^qdt
\\=&0.
\end{align*}
Therefore,
$$\lim_{R\rightarrow T}\|(Tu)\chi_{E_{R}}\|_{L^{q,\mu}_+((0,T),\omega^q)}=0$$
uniformly in $Tu\in T(\mathbb{K})$.

Consequently, by Theorem \ref{Theorem 4.2}, we conclude that $T(\mathbb{K})$ is a  relative compact set in ${L^{q,\mu}_+((0,T),\omega^q)}$. This completes the proof of Proposition \ref{prop1}.
\end{proof}

\section{Representation of functions by fractional integrals in locally one-sided weighted Morrey spaces }
The integral equation
\begin{equation}
I_{0^{+}}^{\alpha}\varphi(x)=f(x),\quad x>0, \label{ae}
\end{equation}
where $0<\alpha<1$, is called {\it Abel's equation} \cite{SKM}.
Denote by
\begin{equation}f_{1-\alpha}(x)=I_{0^{+}}^{1-\alpha}f(x)=\frac{1}{\Gamma(1-\alpha)}\int_{0}^{x}\frac{f(t)dt}{(x-t)^{\alpha}}.
\label{f1-}
\end{equation}

The solvable condition of \eqref{ae} is usually associated with the
 absolutely continuous functions in continuous or Lebesgue spaces.
\begin{definition}
A function $f(x)$ is called absolutely continuous on an interval $\Omega$, if for every $\v>0$ there exists $\delta>0$
such that for any finite set of pairwise disjoint intervals $[a_k, b_k]\subset \Omega$, $k=1,2,\cdots,n,$ such that
$\sum_{k=1}^{n}(b_k-a_k)<\delta$, the inequality $\sum_{k=1}^{n}|f(b_k)-f(a_k)|<\varepsilon$ holds. The space of these functions is
denoted by $AC(\Omega)$.
\end{definition}

It is well known \cite{SKM} that the space $AC(\Omega)$ coincides with the space of primitives of Lebesgue summable functions:
\begin{equation}
f(x)\in AC([a,b])\Leftrightarrow f(x)=c+\int_{a}^{x}\varphi(t)dt,\quad \int_{a}^{b}|\varphi(t)|dt<\infty.
\label{ac}
\end{equation}

Samko et al. \cite{SKM} introduced the following representation of functions by fractional integrals in Lebesgue spaces.

\begin{theorem}
Let $0<\alpha<1$. Abel equation \eqref{ae} is solvable in ${L}^{1}(0,T)$ if and only if
\begin{equation}
f_{1-\alpha}(x)\in AC([0,T]), \quad f_{1-\alpha}(0)=0.\label{Ae}
\end{equation}
These conditions being satisfied the equation has a unique solution given by
\begin{equation}
\varphi(x)=\frac{1}{\Gamma(1-\alpha)}\frac{d}{dx}\int_{0}^{x}\frac{f(t)dt}{(x-t)^{\alpha}}=D_{0^{+}}^{\alpha}f(x).\label{so}
\end{equation}
\end{theorem}
For $\w \in A_1^+$, $0<\lambda<1,$ the corresponding characterization of the solution to \eqref{ae} in locally one-sided weighted Morrey
space ${L}^{1,\lambda}_+((0,T),\omega)$  also holds.

\begin{theorem}\label{Theorem 5.2}
Let $0<\lambda<1$ and $\w \in A_1^+$. Abel equation \eqref{ae} is solvable in ${L}^{1,\lambda}_+((0,T),\omega)$ if and only if
\eqref{Ae} holds and $$ f'_{1-\alpha}(x)\in {L}^{1,\lambda}_+((0,T),\omega),$$
where $f'_{1-\alpha}(x)=\frac{d}{dx}f_{1-\alpha}(x)$. These conditions being satisfied the equation has a unique solution given by
\eqref{so}.
\end{theorem}

\begin{proof}
The proof is quite similar to that of Theorem 2.1 in \cite{SKM}. Assume that \eqref{ae} is solvable in ${L}^{1,\lambda}_+((0,T),\omega)$. By the Fubini theorem, we get
$$\int_{0}^{x}\varphi(\tau)d\tau\int_{\tau}^{x}\frac{dt}{(x-t)^{\alpha}(t-\tau)^{1-\alpha}}
=\Gamma(\alpha)\int_{0}^{x}\frac{f(t)dt}{(x-t)^{\alpha}}.$$
Note that
$$\int_{0}^{x}\frac{dt}{(x-t)^{\alpha}}\int_{0}^{t}\frac{\varphi(\tau)}{(t-\tau)^{1-\alpha}}d\tau
=\Gamma(\alpha)\int_{0}^{x}\frac{f(t)dt}{(x-t)^{\alpha}}.$$
Therefore,
$$B(\alpha,1-\alpha)\int_{0}^{x}\varphi(\tau)d\tau
=\Gamma(\alpha)\int_{0}^{x}\frac{f(t)dt}{(x-t)^{\alpha}}.$$
Then
\begin{equation}
\int_{0}^{x}\varphi(\tau)d\tau
=\frac{1}{\Gamma(1-\alpha)}\int_{0}^{x}\frac{f(t)dt}{(x-t)^{\alpha}}=f_{1-\a}(x).
\label{f1}
\end{equation}
If $\varphi\in {L}^{1,\lambda}_+((0,T),\omega)$, we get
$$\frac{1}{T^{\lambda-1}{\omega(-T,0)}}\int_{(0,T)}|\varphi(\tau)|d\tau<+\infty,$$
which implies $\varphi\in {L}^{1,\lambda}_+((0,T),\omega)\subset L^{1}(0,T)$. By \eqref{ac},
$f_{1-\alpha}(x)\in AC([0,T])$ and $f_{1-\alpha}(0)=0$.

After differentiation in \eqref{f1}, we have
\begin{equation}
\varphi(x)=\frac{1}{\Gamma(1-\alpha)}\frac{d}{dt}\int_{0}^{x}\frac{f(t)dt}{(x-t)^{\alpha}}=f'_{1-\alpha}(x).
\label{22}
\end{equation}
Thus $f'_{1-\alpha}(x)\in {L}^{1,\lambda}_+((0,T),\omega).$

On the other hand, from $f'_{1-\alpha}(x)\in {L}^{1,\lambda}_+((0,T),\omega)$ we know that the function $\varphi(x)$ given by
\eqref{22} exists and $\varphi(x)\in {L}^{1,\lambda}_+((0,T),\omega)$. Next we need to show that it is indeed a solution of \eqref{ae}.
Substitute $\varphi(x)$ into the left hand side of \eqref{ae} and denote the result by $g(x)$, i.e.
\begin{equation}
\frac{1}{\Gamma(\alpha)}\int_{0}^{x}\frac{\varphi(t)}{(x-t)^{1-\alpha}}dt=
\frac{1}{\Gamma(\alpha)}\int_{0}^{x}\frac{f'_{1-\alpha}(t)}{(x-t)^{1-\alpha}}dt=:g(x).\label{210}
\end{equation}
It suffices to show that $g(x)=f(x)$. Note that \eqref{210}
is an equation of type \eqref{ae} with respect to $f'_{1-\alpha}(x)$. Using \eqref{22}, we get
\begin{equation*}
f'_{1-\alpha}(x)=\frac{1}{\Gamma(1-\alpha)}\frac{d}{dx}\int_{0}^{x}\frac{g(t)dt}{(x-t)^{\alpha}}=g'_{1-\alpha}(x).
\end{equation*}
Since $f_{1-\alpha}$ is absolutely continuous, and by \eqref{f1} with $g(x)$
substituting for $f(x)$ on the right hand side, we deduce that $g_{1-\alpha}$ is also absolutely continuous. Then
$$f_{1-\alpha}(x)-g_{1-\alpha}(x)=c.$$
Since \eqref{210} is a solvable equation, we get $g_{1-\alpha}(0)=0$ . This together with $f_{1-\alpha}(0)=0$ leads to $c=0$. Hence
$$\int_{0}^{x}\frac{f(t)-g(t)}{(x-t)^{\alpha}}dt=0,$$
which is also an equation of type \eqref{ae}. The uniqueness of its solution implies that $f(x)-g(x)=0$.
This completes the proof of Theorem \ref{Theorem 5.2}.
\end{proof}

\section{{Nonlinear fractional differential equations in locally one-sided weighted Morrey spaces}}
Consider the classical Cauchy problem for the nonlinear fractional differential equation:
\begin{equation}
\begin{cases}
D_{0^{+}}^{\alpha}u(t)=f(t,u(t)),\\
I_{0^{+}}^{1-\alpha}u(0)=0.\label{eq}
\end{cases}
\end{equation}
The initial condition $I_{0^{+}}^{1-\alpha}u(0)=0$ in \eqref{eq} is (more or less) equivalent to the following initial (weighted)
condition:
$$\lim_{t\rightarrow 0^+}t^{1-\alpha}u(t)=0.$$
For more details, see \cite{KST}. Kilbas and Trujillo \cite{KST} established the following relation between the fractional integration operator $I_{0^{+}}^{\alpha}$ and the
fractional differentiation operator $D_{0^+}^{\alpha}$.
\begin{lemma}[\cite{KST}]\label{th:lem0}
Let $0<\alpha<1$.
 If $f(x)\in L^1(0,T)$ and $f_{1-\alpha}(x)\in AC([0,T])$, then the equality
$$\left(I_{0^{+}}^{\alpha}D_{0^{+}}^{\alpha}f\right)(x)=f(x)-\frac{f_{1-\alpha}(0)}{\Gamma(\alpha)}x^{\alpha-1}$$
holds almost everywhere on $[0,T]$, where $f_{1-\alpha}(x)$ is defined in \eqref{f1-}.
\end{lemma}
Next, we recall the famous Schauder fixed-point theorem.
\begin{lemma}[\cite{S}]\label{le 6.1}
Let $H$ be a convex and closed subset of a Banach space. Then any continuous
and compact map $T : H \rightarrow H$ has a fixed point.
\end{lemma}

By fixed-point theorem above, we establish the existence and uniqueness of solutions to the Cauchy problem \eqref{eq} in locally one-sided weighted Morrey spaces.

\begin{theorem}\label{th:thm42}
Let $0<\sigma<1$, $1<p, q<\infty$, $0\leq\beta,\mu<1$, $1/p-1/q=\sigma$, $\omega\in A_{(p,q)}^+$ and $q+\mu\leq2$. Suppose the operator $F:F(u)=f(t,u(t))$ is bounded, continuous
from ${L}^{q,\mu}_+((0,\delta),\omega^q)$ to ${L}^{p,\beta}_+((0,\delta),\omega^p)$.
If $\frac{1}{2}(1+\frac{1}{p})<\alpha<1$, then the Cauchy problem \eqref{eq} has at least a solution $u\in {L}^{q,\mu}_+((0,\delta),\omega^q)$
for a sufficiently small $\delta$.
Furthermore, if there exists a constant $C_F>0$ such that
\begin{equation}
\|Fu-Fv\|_{L^{p,\beta}_+((0,\delta),\omega^p)}\leq C_F\|u-v\|_{L^{q,\mu}_+((0,\delta),\omega^q)},\quad u,v\in L^{q,\mu}((0,\delta),\omega^q),
\label{CF}
\end{equation}
then the solution of \eqref{eq} is unique in ${L}^{q,\mu}((0,\delta),\omega^q)$ for a sufficiently small $\delta$.
\end{theorem}

\begin{proof}
Since $f(t,u(t))\in {L}^{p,\beta}_+((0,\delta),\omega^p)\subseteq L(0,\delta)$ and $D_{0^{+}}^{\alpha}u(t)=\frac{d}{dt}I_{0^{+}}^{1-\alpha}u(t)$, by \eqref{ac}, we have
$$u_{1-\alpha}=I_{0^{+}}^{1-\alpha}u(t)\in AC([0,\delta]).$$ By Lemma \ref{th:lem0},
 the derivative equation \eqref{eq} in ${L}^{q,\mu}_+((0,\delta),\omega^q)$ is equivalent to the following integral equation:
\begin{equation}\begin{split}
u(t)&=\begin{cases}\frac{1}{\Gamma(\alpha)}\int_{0}^{t}\frac{f(\tau,u(\tau))}{(t-\tau)
^{1-\alpha}}d\tau,\quad t\leq \delta\\
0,\quad\quad \quad\quad\quad \quad\,\, a.e.\quad t\geq\delta,
\end{cases}\\
&=\frac{\chi_{(0,\delta)}(t)}{\Gamma(\alpha)}\int_{0}^{t}\frac{f(\tau,u(\tau))}
{(t-\tau)^{1-\alpha}}d\tau=:T(f(t,u)).\label{eq1}
\end{split}\end{equation}
Set
$$Au(t):=T(f(t,u))=\frac{\chi_{(0,\delta)}(t)}{\Gamma(\alpha)}\int_{0}^{t}\frac{f(\tau,u(\tau))}
{(t-\tau)^{1-\alpha}}d\tau.$$
Then the equation \eqref{eq1} has a solution in ${L}^{q,\mu}_+((0,\delta),\omega^q)$
if and only if the operator $A$ has a fixed point in ${L}^{q,\mu}_+((0,\delta),\omega^q)$.

Next, we show that $A$ is completely continuous. Recalling from Proposition \ref{prop1}, we deduce that $T$ is a compact operator from ${L}^{p,\beta}_+((0,\delta),\omega^p)$ to ${L}^{q,\mu}_+((0,\delta),\omega^q)$. Since
$F:u\rightarrow f(t,u)$ is bounded and continuous from ${L}^{q,\mu}_+((0,\delta),\omega^q)$ to ${L}^{p,\beta}_+((0,\delta),\omega^p)$ and $Au(t)=TFu(t)$, we conclude that $A$ is a compact operator from
${L}^{q,\mu}_+((0,\delta),\omega^q)$ to ${L}^{q,\mu}_+((0,\delta),\omega^q)$ and is
also continuous. Hence, $A:{L}^{q,\mu}_+((0,\delta),\omega^q)\rightarrow {L}^{q,\mu}_+((0,\delta),\omega^q)$
is completely continuous.

Choose a positive constant $0<M^*<\infty$, set $D:=\big\{u:\|u\|_{{L}^{q,\mu}_+((0,\delta),\omega^q)}\leq M^*\big\}$. Then $D$ is a bounded closed convex set.
For any $0\leq x_0\leq \delta$ and $0<h\leq \delta$, we have
\begin{align*}
&\frac{1}{\Phi_{\mu,q}^{\omega,+}(x_0,h)}\int_{(x_0,x_0+h)\bigcap(0,\delta)}|Au(t)|^qdt\\
&\quad\leq\frac{1}{\Phi_{\mu,q}^{\omega,+}(x_0,h)\Gamma(\alpha)^{q}}\int_{(x_0,x_0+h)
\bigcap(0,\delta)}
\left|\int_{0}^{t}\frac{f(\tau,u(\tau))}{(t-\tau)^{1-\alpha}}d\tau\right|^qdt\\
&\quad=:\mathcal{K}.
\end{align*}
For $h\leq\delta$, there exists a nonnegative integer $K$ such that $\delta/2^{k+1}<h\leq \delta/2^{k}$. By the H\"older inequality, we get
\begin{align*}
\mathcal{K}
\leq&
\frac{1}{\Phi_{\mu,q}^{\omega,+}(x_0,h)\Gamma(\alpha)^q\big(\frac{\alpha p-1}{p-1}\big)^{q/p'}}\int_{(x_0,x_0+h)\bigcap(0,\delta)}\bigg[\int_{0}^{\delta}|f(\tau,u(\tau))|^pd
\tau\bigg]^{\frac{q}{p}}t^{\frac{(\alpha p-1)q}{p}}dt\\
\leq& \frac{\delta^{{{{(\alpha p-1+\beta)q/p}}}-1}\|Fu\|_{L^{p,\beta}_+((0,\delta),\omega^p)}^q}{\big(\frac{\alpha p-1}{p-1}\big)^{q/p'}\Gamma(\alpha)^qh^{\mu-2}}\frac{\omega^q(x_0-2\delta,x_0)}
{\omega^q(x_0-h,x_0)}\\
\leq& \frac{C\|F\|^qR^q}{\Gamma(\alpha)^q\left(\frac{\alpha p-1}{p-1}\right)^{q/p'}}\delta^{\frac{(\alpha p+\beta-1)q}{p}+1-\mu}.
\end{align*}
Then
$$\|Au\|_{L^{q,\mu}_+((0,\delta),\omega^q)}\leq \frac{C\|F\|R}{\Gamma(\alpha)\big(\frac{\alpha p-1}{p-1}\big)^{1/p'}}\delta^{\frac{\alpha p+\beta-1}{p}+\frac{1-\mu}{q}}.$$
Set
$$\delta=\bigg[\frac{\Gamma(\alpha)\big(\frac{\alpha p-1}{p-1}\big)^{1/p'}}
{C\|F\|}\bigg]^{\left(\frac{\alpha p+\beta-1}{p}+\frac{1-\mu}{q}\right)^{-1}}.$$
Hence
$$\|Au\|_{L^{p,\beta}_+((0,\delta),\w^p)}\leq M^*.$$
That is to say $A:D\rightarrow D$. It follows from Lemma \ref{le 6.1} that $A$ has a fixed point in $D$. Therefore, Equation \eqref{eq} has at least a
solution in  ${L}^{q,\mu}_+((0,\delta),\omega^q)$.

In addition, if \eqref{CF} holds, then for $u_1,~u_2\in {L}^{q,\mu}_+((0,\delta),\omega^q)$, by the similar
discussions as above, we have
\begin{equation}\begin{split}
\|Au_1-Au_2\|_{L^{q,\mu}_+((0,\delta),\omega^q)}
&\leq \frac{C\|Fu_1-Fu_2\|_{L^{q,\mu}_+((0,\delta),\omega^q)}}{\Gamma(\alpha)(\frac{\alpha q-1}{q-1})^{1/p'}}\delta^{\frac{\alpha p-1}{p}+\frac{1-\mu}{q}}\\
&\leq \frac{CC_F}
{\Gamma(\alpha)(\frac{\alpha p-1}{p-1})^{1/p'}}\delta^{\frac{\alpha p-1}{p}+\frac{1-\mu}{q}}
\|u_1-u_2\|_{L^{q,\mu}((0,\delta),\omega^q)}.
\end{split}\end{equation}
Set
$$\delta<\bigg[\frac{\Gamma(\alpha)\big(\frac{\alpha q-1}{q-1}\big)^{1/p'}}
{CC_F}\bigg]^{\left[\frac{\alpha p-1}{p}+\frac{1-\mu}{q}\right]^{-1}}.$$
Then
$A$ is a contraction mapping in ${L}^{q,\mu}_+((0,\delta),\omega^q)$ and has a unique fixed point in ${L}^{q,\mu}_+((0,\delta),\omega^q)$. This implies that the Cauchy problem \eqref{eq} has a unique solution in ${L}^{q,\mu}_+((0,\delta),\omega^q)$.
\end{proof}
\begin{remark}\label{Rem2.1}\,\
The conditions that operator $F:F(u)=f(t,u(t))$ is bounded, continuous
from ${L}^{q,\mu}_+((0,\delta),\omega^q)$ to ${L}^{p,\beta}_+((0,\delta),\omega^p)$, $\frac{1}{2}(1+\frac{1}{p})<\alpha<1$ are sufficient but not necessary.
\end{remark}
Indeed, for the following differential equation of fractional order $0<\a<1$
(see \cite[Example $4.4$]{FTW}),
\begin{equation}\label{equ 1}\,\
\begin{cases}
D_{0^+}^{\a}u(t)=\lambda t^{\g}(u(t))^2,\\
I_{0^+}^{1-\a}u(0)=0,
\end{cases}
\end{equation}
where $t>0$, $\lambda,\g \in \mathbb{R}$ and $\lambda\neq0$. Note that equation $(\ref{equ 1})$ has the exact solution
\begin{equation*}
u(t)=\begin{cases}
\frac{\Gamma (1-\a-\g)}{\lambda\Gamma (1-2\a-\g)}t^{-(\a+\g)},\quad 0<t<\delta,\\
0,\quad \rm{a.e.}  \emph{t}\geq\d,
\end{cases}
\end{equation*}
where $0<\a+\g<1$. Moreover, in this case, we also have
\begin{equation*}
f(t,u(t))=\begin{cases}
\frac{1}{\lambda}\big[\frac{\Gamma (1-\a-\g)}{\Gamma (1-2\a-\g)}\big]^2t^{-(2\a+\g)},\quad 0<t<\delta,\\
0,\quad \rm{a.e.}  \emph{t}\geq\d.
\end{cases}
\end{equation*}
Take $\w=|x|^{-{1/(2q)}}\in A_{(p,q)}^+$, we claim that $u\in {L}^{q,\mu}_+((0,\delta),|x|^{-{1/2}})$ for $1-(\a+\g)q-\mu>0$, $\m\geq 1/2$. Meanwhile, we also obtain that $f(t,u(t))\notin {L}^{p,\beta}_+((0,\delta),|x|^{-{p/(2q)}})$ for $2(\a+\g)p>1$.

For $0\leq x_0\leq \d$, $0<h\leq \delta$, we have
\begin{align*}
&\frac{1}{\Phi_{\mu,q}^{\omega,+}(x_0,h)}\int_{(x_0,x_0+h)\bigcap(0,\d)}|u(t)|^qdt\\
&\quad\lesssim\frac{1}{{h^{\mu-1}}\int_{x_0-h}^{x_0}|x|^{-{1/2}}dx}\int_{(x_0,x_0+h)
\bigcap(0,\delta)}t^{-(\a+\g)q}dt\\
&\quad=:H_2.
\end{align*}
If $x_0>h>0$, then
\begin{align*}
H_2\lesssim&\frac{1}{{h^{\mu-1}}(\sqrt{x_0}-\sqrt{x_0-h})}
\int_{x_0}^{x_0+h}t^{-(\a+\g)q}dt\\
\lesssim& \frac{\sqrt{x_0}+\sqrt{x_0-h}}{{h^{\mu}}}\left[(x_0+h)^{1-(\a+\g)q}-x_0^{1-(\a+\g)q}\right]\\
\lesssim& \d^{\frac{3}{2}-(\a+\g)q-\mu}< \infty.
\end{align*}
If $ x_0\leq h \leq \delta$, then
\begin{align*}
H_2\lesssim &\frac{1}{{h^{\mu-1}}(\sqrt{x_0}+\sqrt{h-x_0})}\int_{x_0}^{x_0+h}t^{-(\a+\g)q}dt\\
\leq& \frac{1}{{h^{\mu-1/2}}}\left[(x_0+h)^{1-(\a+\g)q}-x_0^{1-(\a+\g)q}\right]\\
\leq& \d^{\frac{3}{2}-(\a+\g)q-\mu}< \infty.
\end{align*}
Combining with the above estimates, we have $u\in {L}^{q,\mu}_+((0,\delta),|x|^{-1/2})$.

On the other hand, we turn to prove $f(t,u(t))\notin {L}^{p,\beta}_+((0,\delta),|x|^{-{p/(2q)}})$. Take $h=\d$ and $x_0=0$. Then
\begin{align*}
&\frac{1}{{\d^{\beta-1}}\int_{-\d}^{0}|x|^{-{p/2q}}dx}\int_{0}^{\delta}\frac{1}{\lambda^p}
\Big[\frac{\Gamma (1-\a-\g)}{\Gamma (1-2\a-\g)}\Big]^{2p}t^{-(2\a+\g)p}dt
\\&\quad\geq\frac{1}{{\d^{\beta-{p/2q}}}}\frac{1}{\lambda^p}
\Big[\frac{\Gamma (1-\a-\g)}{\Gamma (1-2\a-\g)}\Big]^{2p}\int_{0}^{\delta}t^{-(2\a+\g)p}dt
\\&\quad=\infty.
\end{align*}
This implies that $f(t,u(t))\notin {L}^{p,\beta}_+((0,\delta),|x|^{-{p/2q}})$.

\noindent {\bf Acknowledgments} The third author would like to express his gratitude to Professor Dachun Yang for his valuable comments and discussion. This work was partially funded by the National Natural Science Foundation of China (Grant Nos. 11671185 and 11771195), the Natural Science Foundation of Shandong Province
(No. ZR2018LA002).

\bibliographystyle{amsplain}

\end{document}